\documentclass[reqno]{amsart}
\usepackage{amsmath,amsbsy,amsfonts}
\usepackage{color}
\usepackage[colorlinks=true,urlcolor=blue,
citecolor=red,linkcolor=blue,linktocpage,pdfpagelabels,
bookmarksnumbered,bookmarksopen]{hyperref}

\usepackage[utf8]{inputenc}
\usepackage{graphicx}
\usepackage{amsmath,latexsym,amssymb}
\newtheorem{theorem}{Theorem}

\newtheorem{lemma}[theorem]{Lemma}
\newtheorem{proposition}[theorem]{Proposition}
\newtheorem{remark}{Remark}[section]

\newcommand{\dd}{\mathrm{d}}

\numberwithin{equation}{section}

\begin{document}
	\title[Existence, regularity, asymptotic decay and radiality]{Existence, regularity, asymptotic decay and radiality of solutions to some extension problems}
	
	\author{H. Bueno}
	\address{H. Bueno and Aldo H. S. Medeiros - Departmento de Matem\'atica, Universidade Federal de  Minas Gerais, 31270-901 - Belo Horizonte - MG, Brazil}
	\email{hamilton.pb@gmail.com and  aldomedeiros@ufmg.br}
	\author{Aldo H. S. Medeiros}
	\author{G. A. Pereira}
	\address{G. A. Pereira - Departmento de Matem\'atica, Universidade Federal de Juiz de Fora, 36036-330 - Juiz de Fora - MG, Brazil}
	\email{gilberto.pereira@ufop.edu.br}

	\thanks{H. Bueno is the corresponding author; Aldo H. S. Medeiros received research grants from CNPq/Brazil.}
	\date{}

\begin{abstract}
Supposing only that $\displaystyle\lim_{t \to 0} \frac{f(t)}{t} = 0$ and $\displaystyle\lim_{t \to \infty} \frac{f(t)}{t^{p}} = 0$, for some $p \in \left(1,\frac{N+1}{N-1}\right)$, we prove that solutions to the extension problem 
\begin{equation*}\left\{
\begin{array}{rcll}
-\Delta u+ m^2u &=& 0, &\mbox{in} \ \ \mathbb{R}^{N+1}_{+} \\ 
-\frac{\partial u}{\partial{x}} (0,y)& =& f(u(0,y)), & y \in \mathbb{R}^{N},
\end{array}\right.
\end{equation*}
and also to the extension Hartree problem
\begin{equation*}
\left\{\begin{aligned}
-\Delta u +m^2u&=0, &&\mbox{in} \ \mathbb{R}^{N+1}_+,\\
-\displaystyle\frac{\partial u}{\partial x}(0,y)&=-V_\infty u(0,y)+\left(\frac{1}{|y|^{N-\alpha}}*F(u(0,y))\right)f(u(0,y)) &&\mbox{in} \ \mathbb{R}^{N}\end{aligned}\right.
\end{equation*}
are radially symmetric in $\mathbb{R}^N$. In the last problem, $V_\infty>0$ is a constant and $F$ the primitive of $f$. Under the same hypotheses, regularity and exponential decay of solutions to the first problem is also proved and, supposing the traditional Ambrosetti-Rabinowitz condition, also existence of a ground state solution.
\end{abstract}
\maketitle

\section{Introduction}\label{intro}
We will denote the point of $\mathbb{R}^{N+1}_{+}$ by pair $(x,y)$, where $x \in (0,\infty)$ and $y=(y_1,...,y_N) \in \mathbb{R}^{N}$.

In this paper we recall various aspects of the extension problem 
\begin{equation}\label{P}\left\{
\begin{array}{rcll}
 -\Delta u+ m^2u &=& 0, &\mbox{in} \ \ \mathbb{R}^{N+1}_{+} \\ 
 -\frac{\partial u}{\partial{x}} (0,y)& =& f(u(0,y)), & y \in \mathbb{R}^{N}.
\end{array}\right.
\end{equation}
Our main result consists in proving that solutions of \eqref{P} are radially symmetric in $\mathbb{R}^N$ with respect to a point $y_0\in\mathbb{R}^N$. We also prove the same result for the solutions to the pseudo-relativistic Hartree problem
\begin{equation}\label{Hartree}
\!\!\!\left\{\begin{aligned}
-\Delta u +m^2u&=0, &&\mbox{in} \ \mathbb{R}^{N+1}_+,\\
-\displaystyle\frac{\partial u}{\partial x}(0,y)&=-V_\infty u(0,y)+\left(\frac{1}{|y|^{N-\alpha}}*F(u(0,y))\right)f(u(0,y)) &&\mbox{in} \ \mathbb{R}^{N}\end{aligned}\right.
\end{equation}
(where $V_\infty>0$ is a constant and $F$ is the primitive of $f$). 

Our hypotheses on $f$ are very mild. We suppose that the $C^1$-nonlinearity $f$ satisfies
\begin{enumerate}
	\item[$(f_1)$] $\displaystyle\lim_{t \to 0} \frac{f(t)}{t} = 0$;
	\item[$(f_2)$] $\displaystyle\lim_{t \to \infty} \frac{f(t)}{t^{p}} = 0$, for some $p \in \left(1,\frac{N+1}{N-1}\right)$.
\end{enumerate}

Under these hypotheses we also show that solutions of \eqref{P} are regular and have exponential decay. Although not new, we understand that the review of these results might be helpful. Consonant with this proposal, the proofs we present in this article are very detailed.

We also address a simple situation of existence of solutions to problem \eqref{P}, supposing additionally that $f$ satisfies
\begin{enumerate}
\item[$(f_3)$] There exist $\theta > 2$ such that
\[0<\theta F(t) < tf(t), \quad \forall t>0,
\]
where $F(t) = \displaystyle\int_{0}^{t}f(s) \dd s$.
\end{enumerate}
Of course, solutions of \eqref{P} can be obtained under much milder assumptions, see e.g., \cite{BMedP,BBMP}. Because we are looking for a positive solution, we suppose that $f(t)=0$ for $t<0$.

Of course, problem \eqref{P} results from the application of the Dirichlet to Neumann operator to the problem
\begin{equation}\label{original}\sqrt{-\Delta+m^2}\, u=f(u)\quad\text{in }\ \mathbb{R}^N,\end{equation}
while \eqref{Hartree} comes from
\begin{equation}\label{originalHartree}\sqrt{-\Delta+m^2}\, u+V_\infty u=\left(\frac{1}{|y|^{N-\alpha}}*F(u)\right)f(u)\ \ \text{in }\ \mathbb{R}^{N}.\end{equation}
Problems \eqref{original} and \eqref{originalHartree} have been extensively studied in recent years, see \cite{Cho,Cingolani,Cingolani2,CSS,ZelatiNolasco,ZelatiNolasco2,Elgart,Frohlich,Lenzmann,Moroz0,Moroz,Wei}. See also the classical paper of Lieb \cite{Lieb2}.

Changing the operator $\sqrt{-\Delta+m^2}$ to its generalization $(-\Delta+m^2)^\sigma$, $0<\sigma<1$, problems like \eqref{original} and \eqref{originalHartree} can be found in \cite{Ambrosio,BMP,VFelli}.
We summarize our results:
\begin{theorem}\label{t1}
Suppose that conditions \textup{($f_1$)-($f_3$)}, are valid. Then, problem \eqref{P} has a non-negative ground state solution $w\in H^1(\mathbb{R}^{N+1}_+)$.
\end{theorem}

\begin{theorem}\label{classical}
	Assuming \textup{($f_1$)-($f_2$)}, any solution $v$ of problem \eqref{P} satisfies
	\[v\in C^{1,\alpha}(\mathbb{R}^{N+1}_+)\cap C^2(\mathbb{R}^{N+1}_+)\]
	and therefore is a classical solution of \eqref{P}.
\end{theorem}
As a simple remark, we observe that if $f\in C^\infty$, then the solution is also $C^\infty$.

We also prove that the ground state solution has exponential decay:
\begin{theorem}\label{t3} Suppose that $v$ is a weak solution to \eqref{P}.
	
Then $v(x,y) > 0$ in $[0,\infty) \times\mathbb{R}^{N}$ and, for any $0\leq \alpha < m$, there exists $C>0$ such that
\[0 < v(x,y) \leq Ce^{-(m-\alpha)\sqrt{x^2 + \vert y \vert^{2}}}e^{-\alpha x}
\]
for any $(x,y) \in \mathbb{R}^{N+1}_{+}$.
	
In particular, there exists $\delta \in (0,m)$ such that
\[0 < v(0,y) \leq Ce^{-\delta \vert y \vert}, \,\, \text{for any} \,\, y \in \mathbb{R}^{N}.
\] 
\end{theorem}

\begin{theorem}\label{t4}
Any solution $v$ to problem \eqref{P} is radially symmetric on $\mathbb{R}^N$ with respect to some $y_0\in\mathbb{R}^N$. 
\end{theorem}

\begin{theorem}\label{t5}
Any solution $v$ to problem \eqref{Hartree} is radially symmetric on $\mathbb{R}^N$ with respect to some $y_0\in\mathbb{R}^N$.  
\end{theorem}

The natural setting for problem \eqref{P} is the Sobolev space \[H^1(\mathbb{R}^{N+1}_+)=\left\{u\in L^2(\mathbb{R}^{N+1}_+)\,:\, \iint_{\mathbb{R}^{N+1}_+}|\nabla u|^2\dd x\dd y<\infty \right\}\]
considered with the norm
\[\|u\|^2=\iint_{\mathbb{R}^{N+1}_+}\left(|\nabla u|^2+u^2\right)\dd x\dd y.\]

\noindent\textbf{Notation.} The norm in the space $\mathbb{R}^{N+1}_+$ will be denoted by $\|\cdot\|$. For all $q\in [1,\infty]$, we denote by $|\cdot|_q$ the norm in the space $L^q(\mathbb{R}^{N})$ and by $\|\cdot\|_q$ the norm in the space $L^{q}(\mathbb{R}^{N+1}_+)$. From now on, integrals in $\mathbb{R}^{N+1}_+$ will be denoted without $\dd x\dd y$. \vspace*{.3cm}

Traces of functions $H^1(\mathbb{R}^{N+1}_+)$ are in  $H^{1/2}(\mathbb{R}^{N})$ and every function in $H^{1/2}(\mathbb{R}^{N})$ is the trace of a function in $H^1(\mathbb{R}^{N+1}_+)$, see \cite{Tartar}. Denoting $\gamma\colon H^1(\mathbb{R}^{N+1}_+)\to H^{1/2}(\mathbb{R}^{N})$ the linear function that associates the trace $\gamma(v)\in H^{1/2}(\mathbb{R}^{N})$ of the function $v\in H^1(\mathbb{R}^{N+1}_+)$, then $\ker\,\gamma=H^1_0(\mathbb{R}^{N+1}_+)$.

The immersions
\begin{align}\label{immersions}H^1(\mathbb{R}^{N+1}_+)&\hookrightarrow L^q(\mathbb{R}^{N+1}_+)\\
H^{1/2}(\mathbb{R}^{N})&\hookrightarrow L^q(\mathbb{R}^{N})\end{align}
are continuous for any $q\in [2,2^*]$ and $[2,2^{\#}]$ respectively, where
\begin{equation}\label{2*}2^{*}=\frac{2(N+1)}{N-1}\qquad\textrm{and}\qquad 2^{\#}=\frac{2N}{N-1}.\end{equation}

For a bounded open set $\Omega\subset\mathbb{R}^{N}$ we have (see \cite{Demengel}),  $H^{1/2}(\mathbb{R}^{N})=W^{1/2,2}(\mathbb{R}^{N})$. We recall the definition of and $W^{1/2,2}(\Omega)$. Let $u\colon \Omega\to \mathbb{R}$ a measurable function and $\Omega$ a bounded open set (that, in the sequel, we suppose to have Lipschitz boundary). Denoting
\[[u]^2_{\Omega}=\int_\Omega\int_\Omega\frac{|u(x)-u(y)|^2}{|x-y|^{N+1}}\dd x\dd y\]
and
\begin{align*}W^{1/2,2}(\Omega)&=\left\{u\in L^2(\mathbb{R}^{N})\,:\,[u]^2_{\Omega}<\infty\right\},
\end{align*}
then $W^{1/2,2}(\Omega)$ is a reflexive Banach space (see, e.g., \cite{Demengel} and \cite{Guide}) endowed with the norm
\[\|u\|_{W^{1/2,2}(\Omega)}=|u|_2+[u]_{\Omega}.\]\goodbreak

The proof of the next result can be found in \cite[Theorem 4.54]{Demengel}.
\begin{theorem}\label{immersionW}
	The immersion $W^{1/2,2}(\Omega)\hookrightarrow L^q(\Omega)$ is compact for any $q\in \left[1,2^{\#}\right)$.
\end{theorem}

As usual, the immersion $W^{1/2,2}(\Omega)\hookrightarrow L^{2^{\#}}(\Omega)$ is continuous: see \cite[Corollary 4.53]{Demengel}. We denote the norm in the space $L^q(\Omega)$ by $|\cdot|_{L^q(\Omega)}$.

Problem \eqref{P} is related to the energy functional 
\[J(u) =\frac{1}{2}\iint_{\mathbb{R}^{N+1}_+}\left(|\nabla u|^2+m^2u^2\right)-\frac{1}{2}\int_{\mathbb{R}^{N}}F(\gamma(u))\]
and, since the derivative of the energy functional is given by
\begin{align}\label{derivative}
J'(u)\cdot \varphi =&\iint_{\mathbb{R}^{N+1}_+}\left[\nabla u\cdot\nabla \varphi+m^2u\varphi\right]-\int_{\mathbb{R}^{N}}f(\gamma(u))\gamma(\varphi)
\end{align}
for all $\varphi\in H^1(\mathbb{R}^{N+1}_+)$, we see that critical points of $J$ are weak solutions \eqref{P}.

\section{Preliminaries}
Let us suppose that  $u\in H^1(\mathbb{R}^{N+1})\cap C^\infty_0(\mathbb{R}^{N+1}_+)$ and $u(x,y)\geq 0$. Let us proceed heuristically: since
\[|u(0,y)|^t=\int_{\infty}^{0}\frac{\partial}{\partial x}|u(x,y)|^t\dd x=\int_{\infty}^{0}t|u(x,y)|^{t-2}u(x,y)\frac{\partial}{\partial x}u(x,y)\dd x,\]
it follows from Hölder's inequality
\begin{align}\label{Heu}\int_{\mathbb{R}^{N}}|\gamma(u)|^t=\int_{\mathbb{R}^{N}}|u(0,y)|^t\dd y&\leq \int_{\mathbb{R}^{N}}\int_0^\infty t|u(x,y)|^{t-1}|\nabla u(x,y)|\nonumber\\
&\leq t\left(\int_{\mathbb{R}^{N+1}_+}|u|^{2(t-1)}\right)^{1/2}\left(\int_{\mathbb{R}^{N+1}_+}|\nabla u|^2\right)^{1/2}\nonumber\\
&\leq t\|u\|_{2(t-1)}^{t-1}\|\nabla u\|_{2}.
\end{align}
So, in order to apply the immersion $H^1(\mathbb{R}^{N+1}_+)\hookrightarrow L^q(\mathbb{R}^{N+1}_+)$ we must have $2\leq 2(t-1)\leq \frac{2(N+1)}{N-1}$, that is, \begin{equation}\label{p}
2\leq t\leq\frac{2N}{N-1}=2^{\#}.
\end{equation}
By density of $H^1(\mathbb{R}^{N+1})\cap C^\infty_0(\mathbb{R}^{N+1}_+)$ in $H^1(\mathbb{R}^{N+1}_+)$, the estimate \eqref{Heu} is valid for all $u\in H^1(\mathbb{R}^{N+1}_+)$.

Taking into account \eqref{immersions}, Young's inequality applied to \eqref{Heu} yields
\begin{align}\label{casep}|\gamma(u)|_{t}&\leq \|u\|_{2(t-1)}^{(t-1)/t}\left(t\|\nabla u\|_{2}\right)^{1/t}\\
&\leq \frac{t-1}{t}\|u\|_{2(t-1)}+\|\nabla u\|_{2}\nonumber\\
&\leq C_t\|u\|,\nonumber
\end{align}
where $C_t$ is a constant. We summarize:
\begin{equation}\label{gammav}
|\gamma(u)|\in {L^t(\mathbb{R}^{N})},\ \ \forall\ t\in [2,2^{\#}].\end{equation}

The inequality \eqref{casep} will also be valuable in the special case $t=2$:
\begin{align}\label{p=2}
|\gamma(u)|^2_{2}&\leq \|u\|_{2}\left(2\|\nabla u\|_{2}\right)\nonumber\\
&\leq \lambda\iint_{\mathbb{R}^{N+1}_+}u^2+\frac{1}{\lambda}\iint_{\mathbb{R}^{N+1}_+}|\nabla u|^2
\end{align}
where $\lambda>0$ is a parameter, the last inequality being a consequence of Young's inequality.

\begin{remark}\label{obs1}
It follows from \textup{($f_1$)} and \textup{($f_2$)} that, for any fixed $\xi>0$, there exists a constant $C_\xi$ such that
\begin{equation}\label{boundf}|f(t)|\leq\xi t+C_\xi t^{p-1},\quad\forall\ t\geq 0\end{equation}
and analogously
\begin{equation}\label{boundF}|F(t)|\leq\xi t^2+C_\xi t^{p}\leq C(t^2+t^p),\quad\forall\ t\geq 0.\end{equation}

Condition $(f_3)$ also yields
\begin{equation}\label{Obs3}
F(t)  \geq C_{0}\vert t \vert^{\theta}, \,\,\, \forall t .
\end{equation}

Observe also that $\gamma(u)\in L^\theta(\mathbb{R}^{N})$ and $\gamma(u)\in L^2(\mathbb{R}^{N})$ imply $F(\gamma(u))\in L^1(\mathbb{R}^{N})$.
\end{remark}

We denote by $L^q_w(\mathbb{R}^{N})$ the weak $L^q(\mathbb{R}^{N})$ space and by $|\cdot|_{q_w}$ its usual norm (see \cite{Lieb}). The next result is a generalized version of the Hardy-Littlewood-Sobolev inequality and will be applied when considering the solutions of the Choquard equation \eqref{Hartree}:
\begin{proposition}[Lieb \cite{Lieb}]\label{pLieb} Assume that $p,q,r\in(1,\infty)$ and \[\frac{1}{p}+\frac{1}{q}+\frac{1}{r}=2.\]
	Then, for some constant $N_{p,q,t}>0$ and for any $f\in L^p(\mathbb{R}^{N})$, $g\in L^r(\mathbb{R}^{N})$ and
	$h\in L^q_w(\mathbb{R}^{N})$, we have the inequality
	\[\int_{\mathbb{R}^{N}}\int_{\mathbb{R}^{N}}f(t)h(t-s)g(s)\dd t\dd s\leq N_{p,q,t}|f|_{p}|g|_{r} |h|_{q_w}.\]
\end{proposition}

\begin{lemma}\label{lemmameas}
	Let $A\subset\mathbb{R}^N$ be a measurable set. Then
	\[\int_{A}u^2(0,y)\dd y\leq \gamma_{0}\iint_{A\times [0,\infty]}\left[|\nabla u|^2+u^2\right].\]
\end{lemma}
\begin{proof}
	Of course, for any $y\in A$, we have
	\begin{align*}
	u^2(0,y)=-\int_0^\infty \frac{\partial}{\partial t}u^2(t,y)\dd y=-\int_0^\infty 2\frac{\partial u}{\partial t}(t,y)u(t,y)\dd t.
	\end{align*}Thus,
	\begin{align*}
	\int_A u^2(0,y)\dd y&=-\int_A\int_0^\infty 2\frac{\partial u}{\partial t}(t,y)u(t,y)\dd t\dd y\\
	&\leq \iint_{A\times [0,\infty]}2 |\nabla u(t,y)|\,|u(t,y)|\,\dd t\dd y\\
	&\leq \iint_{A\times [0,\infty]}\left[|\nabla u(t,y)|^2+u^2(t,y)\right].
	\end{align*}
$\hfill\Box$\end{proof}

\section{Proof of Theorem \ref{t1}}\label{mpg}
\begin{lemma}\label{gpm}
	The functional $J$ satisfies the mountain pass theorem geometry. More precisely,
	\begin{enumerate}
		\item [$(i)$] There exist $\rho,\delta>0$ such that $I|_S\geq \delta>0$ for all $u\in S$, where
		\[S=\left\{u\in H^1(\mathbb{R}^{N+1}_+)\,:\, \|u\|=\rho\right\};\]
		\item [$(ii)$] For any $u_0\in H^1(\mathbb{R}^{N+1}_+)$ such that $\gamma(u_0) \neq 0$ we have $J(\tau u_0) \to -\infty$ as $\tau \to \infty$. 
	\end{enumerate}
\end{lemma}
\begin{proof}Remark \eqref{obs1} and the Sobolev embedding yield
	\begin{align*}
	J(u) &= \frac{1}{2} \iint_{\mathbb{R}^{N+1}_{+}}\left(\vert \nabla u(x,y) \vert^{2} + m^2\vert u(x,y) \vert^{2}\right)- \int_{\mathbb{R}^{N}}F(\gamma(u)) \dd y \\
	& \geq \frac{1}{2} \min\{1,m^2\} \Vert u \Vert^{2} - \varepsilon \int_{\mathbb{R}^{N}} \vert \gamma(u)\vert^2 \dd y -C_{\varepsilon} \int_{\mathbb{R}^{N}} \vert \gamma(u) \vert^{p+1} \dd y\\
	&\geq \left(\frac{C}{2} - C_{1}\varepsilon \right) \Vert u \Vert^{2} - \tilde{C} \Vert u \Vert^{p+1}
	\end{align*}
	
	Taking $0 < \varepsilon < \frac{C}{2C_1}$, for some $a, A>0$ we obtain
	\[
	J(u) \geq a\Vert u \Vert^{2} - A\Vert u \Vert^{p+1}, \,\,\, \text{for all} \,\,\, u \in H^{1}(\mathbb{R}^{N+1}_{+}).
	\]
	Since $p \in \left(1, \frac{N+1}{N-1}\right)$, condition ($i$) is proved. 
	
	In order to prove $(ii)$, fix $u_0 \in H^{1}(\mathbb{R}^{N+1}_{+})$ with $\gamma(u_0) \neq 0$. Thus,
	\[
	J(\tau u_0) \leq \frac{1}{2} \max\{1,m^2\} \tau^{2}\Vert u_0 \Vert^{2} - \int_{\mathbb{R}} F(\gamma (\tau u_0) )\dd y \\
	\leq C \tau^{2}\Vert u_0 \Vert^{2} - C_1\tau^{\theta} \Vert u_0 \Vert^{\theta},
	\]
	the last inequality being a consequence of the Ambrosetti-Rabinowitz condition ($f_3$). Therefore, since $\theta >2$ we obtain
$J(\tau u_0) \to -\infty$ when $\tau \to \infty$, completing the proof.
$\hfill\Box$\end{proof}\vspace*{.4cm}

The existence of a Palais-Smale sequence $(u_n)\subset H^1(\mathbb{R}^{N+1}_+)$ such that
\[J'(u_n)\to 0\qquad\textrm{and}\qquad J(u_n)\to c,\]
where
\begin{align*}c&=\inf_{\alpha\in \Gamma}\max_{t\in [0,1]}J(\alpha(t)),\\
\intertext{and}
\Gamma&=\left\{\alpha\in C^1\left([0,1],H^1(\mathbb{R}^{N+1}_+)\right)\,:\,\alpha(0)=0,\,\alpha(1)<0\right\}
\end{align*} 
is a consequence of the mountain pass theorem without the PS-condition. It is well-known an alternative characterization of the minimax value $c$, see \cite{Rabinowitz} for details,
\begin{align}\label{charac}
c=\inf_{u\in H^1(\mathbb{R}^{N+1}_+)\setminus\{0\}}\max_{t\geq 0}J(tu) > 0.\end{align}
 \vspace*{.2cm}

\noindent\textit{Proof of Theorem \ref{t1}.} Let $(u_n)\subset H^1(\mathbb{R}^{N+1}_+)$ be a sequence such that $J(u_n)\to c$ and $J'(u_n)\to 0$, with $c$ given by \eqref{charac}. Since we have, for all $n$ sufficiently large, 
\begin{align*}
c + \epsilon + \Vert u_n \Vert &\geq J(u_n)- \frac{1}{\theta} J'(u_n)\cdot u_n \\
&= \left(\frac{1}{2} - \frac{1}{\theta}\right) \iint_{\mathbb{R}^{N+1}_{+}}\left(\vert \nabla u_n \vert^{2} + m^2\vert u_n \vert^{2}\right)\\ &\quad + \frac{1}{\theta}\int_{\mathbb{R}^{N}} \left[\theta F(\gamma(u_n)) -  f(\gamma(u_n))\gamma(u_n)\right] \dd y  \\
&\geq \left(\frac{1}{2} - \frac{1}{\theta}\right)C \Vert u_n \Vert^{2}
\end{align*}
and $\theta >2$, we conclude  $(u_n)$ is bounded in $H^{1}(\mathbb{R}^{N+1}_{+})$. 

Thus, for a subsequence 
\[u_n \rightharpoonup u \,\,\, \text{in} \,\, H^{1}(\mathbb{R}^{N+1}_{+}) \,\,\, \text{and} \,\,\, \gamma(u_n) \rightharpoonup \gamma(u) \,\,\, \text{in} \,\,\, L^{q}(\mathbb{R}^{N}),
\]
for all $q \in \left[2,2^{\#} \right]$.
	
As a consequence, for all $v \in H^{1}(\mathbb{R}^{N+1}_{+})$ we have
\[J'(u)\cdot v = \lim_{n \to \infty} J'(u_n)\cdot v = 0.
\]
	

\noindent\textbf{Claim.} There exists a sequence $(y_n)$ in $\mathbb{R}^{N}$ and $\beta > 0$ such that,
\[\int_{B_{1}(y_n)} \vert \gamma(u_n) \vert^{2} \dd y > \beta, \,\,\,\,\, \forall n \in \mathbb{N}.\]
	
In fact, assume that
\[\limsup_{n\to\infty}\sup_{y\in\mathbb{R}^N}\int_{B_1(y)}|u_n|^2\dd x = 0.
\]
	
Thus $u_n\to 0$ in $L^q(\mathbb{R}^N)$ for any $q\in (2,2^{*}_s)$ (see \cite[Lemma 1.2.1]{Willem}). Since $(u_n)$ is bounded in $L^2(\mathbb{R}^N)$, there exists $M_0>0$ such that $|u_n|_2\leq M$ for all $n\in\mathbb{N}$. So, for any $\eta>0$, by taking $\epsilon=\eta/M_0$, we conclude that
\[\int_{ \mathbb{R}^N}|F(u_n)|\dd x\leq \epsilon |u_n|^2_2+C_\epsilon |u_n|^{p+1}_{p+1}=\eta+C_\epsilon |u_n|^{p+1}_{p+1}\]
and since $p \in \left(1,2_{\#}\right)$ and $|u_n|_{p+1} \to 0$ we obtain
\[\int_{ \mathbb{R}^N}F(u_n)\dd x\to 0,\ \text{when }\ n\to\infty.\]
	
Similarly,
\[\int_{ \mathbb{R}^N}f(u_n)u_n\dd x\to 0,\ \text{ when }\ n\to \infty.\]
	
Consequently, when $n\to\infty$,
\[\iint_{\mathbb{R}^{N+1}_{+}}\left(\vert \nabla u_n \vert^{2} + m^2\vert u_n \vert^{2}\right) = J'(u_n)\cdot u_n + \int_{\mathbb{R}^{N}} f(\gamma(u_n))\gamma(u_n) \dd y  \rightarrow 0, 
\]
from what follows
\[0 < c = \lim_{n \to \infty}J(u_n) 
= 0,\]
reaching a contradiction that proves the Claim.
	
Therefore, there exists $\beta >0$ and a sequence $(y_n)$ such that, for all $n\in\mathbb{N}$,
\begin{align}\label{wneq0}
\int_{B_1(y_n)}|\gamma(u_n)|^2\dd x \geq \beta>0.
\end{align}
	
Now, we define $w_n(x,y)=u_n(x, y+y_n)$. Then $\|w_n\|=\|u_n\|$, $J(w_n)=J(u_n)$ and $J'(w_n)\to 0$ when $n\to\infty$. Passing to a subsequence if necessary, we can suppose that, for $q\in [2,2^{*}_s)$, we have
\begin{align*}&w_n\rightharpoonup w\quad \text{in }\ \ H^{1}(\mathbb{R}^{N+1}_{+}),\qquad \gamma(w_n)\to \gamma(w)\ \ \text{in }\ L^q_{loc}(\mathbb{R}^N)\\ \intertext{and}
&\gamma(w_n)(y)\to \gamma(w)(y)\ \ \text{a.e. in }\ \mathbb{R}^N.
\end{align*}
	
Note that
\begin{align*}
\int_{B_{1}(0)} \vert \gamma (w) \vert^{2} \dd y = \lim_{n \to \infty}\int_{B_{1}(0)} \vert \gamma (w_n) \vert^{2} \dd y = \lim_{n \to \infty}\int_{B_{1}(y_n)} \vert \gamma (u_n) \vert^{2} \dd y \geq \beta > 0,
\end{align*}
that is, $\gamma(w) \neq 0$.
	
Furthermore, for all $v \in H^{1}(\mathbb{R}^{N+1}_{+})$, we have
\begin{align*}
J'(w)\cdot v &=\lim_{n\to\infty}\left[\iint_{\mathbb{R}^{N+1}_{+}}\left(\nabla w_n \nabla v + m^2 u w_n v \right) -\int_{ \mathbb{R}^N}f(\gamma(w_n)) \gamma (v) \dd x\right]\\
&=\lim_{n\to\infty} J'(w_n)\cdot v =0,
\end{align*}
and we conclude that $J'(w)=0$. 

We now turn our attention to the positivity of $w$. Seeing that
\[\iint_{\mathbb{R}^{N+1}_+}\left(\nabla w\cdot \nabla v+m^2wv\right)=\int_{\mathbb{R}^{N}}f(\gamma(w))\gamma(v)\]
and choosing $v=w^-$, the left-hand side of the equality is positive,  while the right-hand side is not positive. The proof is complete.
$\hfill\Box$
\section{Proof of Theorem \ref{classical}}
Following arguments in \cite{ZelatiNolasco}, we have:
\begin{lemma}\label{c1} For all $\theta\in \left(2,\frac{2N}{N-1}\right)$, we have $|\gamma(u)|^{\theta-2}\leq 1+g_2$,	where $g_2\in L^N(\mathbb{R}^{N})$.
\end{lemma}

\noindent\begin{proof}We have
\[|\gamma(u)|^{\theta-2}=|\gamma(u)|^{\theta-2}\chi_{\{|\gamma(u)|\leq 1\}}+|\gamma(u)|^{\theta-2}\chi_{\{|\gamma(u)|>1\}}\leq 1+g_2,\]
with $g_2=|\gamma(u)|^{\theta-2}\chi_{\{|\gamma(u)|>1\}}$. If $(\theta-2)N\leq 2$, then
\[\int_{\mathbb{R}^{N}}|\gamma(u)|^{(\theta-2)N}\chi_{\{|\gamma(u)|>1\}}\leq \int_{\mathbb{R}^{N}}|\gamma(u)|^2\chi_{\{|\gamma(u)|>1\}}\leq\int_{\mathbb{R}^{N}}|\gamma(u)|^2<\infty.\]
	
When $2<(\theta-2)N$, then $(\theta-2)N\in \left(2,\frac{2N}{N-1}\right)$ and $|\gamma(u)|^{\theta-2}\in L^N(\mathbb{R}^{N})$ as an outcome of \eqref{gammav}.
	$\hfill\Box$\end{proof}

The proof of the next result adapts arguments in \cite{Cabre} and \cite{ZelatiNolasco}.
\begin{proposition}\label{p1} For all $p\in [2,\infty)$ we have $\gamma(v)\in L^p(\mathbb{R}^{N})$.
\end{proposition}

\noindent\begin{proof}
Choosing $\varphi=\varphi_{\beta,T}=vv^{2\beta}_T$ in \eqref{derivative}, where $v_T=\min\{v_+,T\}$ and $\beta>0$, we have $0\leq \varphi_{\beta,T}\in H^1(\mathbb{R}^{N+1}_+)$ and
\begin{align}\label{varphibetaT}
\iint_{\mathbb{R}^{N+1}_+}\left[\nabla v\cdot\nabla \varphi_{\beta,T}+m^2v\varphi_{\beta,T}\right]
=\int_{\mathbb{R}^{N}}f(\gamma(v))\gamma(\varphi_{\beta,T})\dd y,
\end{align}
Since $\nabla\varphi_{\beta,T}=v^{2\beta}_T\nabla v+2\beta vv^{2\beta-1}_T\nabla v_T$,
the left-hand side of \eqref{varphibetaT} is given by
\[
\iint_{\mathbb{R}^{N+1}_+}\nabla v\cdot \left(v^{2\beta}_T\nabla v+2\beta vv^{2\beta-1}_T\nabla v_T\right)+m^2v\left(vv^{2\beta}_T\right)\]
\begin{equation}\label{varphibetaTl}=\iint_{\mathbb{R}^{N+1}_+}v^{2\beta}_T\left[|\nabla v|^2+m^2v^2\right]+2\beta\iint_{D_T}v^{2\beta}_T|\nabla v|^2,
\end{equation}
where $D_T=\{(x,y)\in (0,\infty)\times \mathbb{R}^{N}\,:\, v_T(x,y)\leq T\}$.

Now we express \eqref{varphibetaTl} in terms of $\|vv^\beta_T\|^2$. For this, we note that $\nabla(vv^\beta_T)=v^\beta_T\nabla v+\beta vv^{\beta-1}_T\nabla v_T$. Therefore,
\[\iint_{\mathbb{R}^{N+1}_+}|\nabla (vv^\beta_T)|^2=\iint_{\mathbb{R}^{N+1}_+}v^{2\beta}_T|\nabla v|^2+(2\beta+\beta^2)\iint_{D_T}v^{2\beta}_T|\nabla v|^2,\]
thus yielding
\begin{align}\label{norm}
\|vv^\beta_T\|^2
&=\left(\iint_{\mathbb{R}^{N+1}_+}v^{2\beta}_T|\nabla v|^2+(2\beta+\beta^2)\iint_{D_T}v^{2\beta}_T|\nabla v|^2\right)+\iint_{\mathbb{R}^{N+1}_+}(vv^\beta_T)^2\nonumber\\
&=\iint_{\mathbb{R}^{N+1}_+}v^{2\beta}_T\left(|\nabla v|^2+|v|^2\right)+2\beta\left(1+\frac{\beta}{2}\right)\iint_{D_T}v^{2\beta}_T|\nabla v|^2\nonumber\\
&\leq C_\beta\left[\iint_{\mathbb{R}^{N+1}_+}v^{2\beta}_T\left(|\nabla v|^2+m^2|v|^2\right)+2\beta\iint_{D_T}v^{2\beta}_T|\nabla v|^2\right],
\end{align}
where $C_\beta=\max\left\{m^{-2},\left(1+\frac{\beta}{2}\right)\right\}$. Gathering \eqref{varphibetaT}, \eqref{varphibetaTl} and \eqref{norm}, we obtain
\begin{align}\label{norm=r}
\|vv^\beta_T\|^2\leq& C_\beta
\int_{\mathbb{R}^{N}}f(\gamma(v))\gamma(v)\gamma(v_T)^{2\beta}.
\end{align}

Since $|f(t)|\leq C_1(|t|+|t|^{\theta-1})$, it follows from \eqref{norm=r}
\begin{align}\label{rhs}
\|vv^\beta_T\|^2&\leq C_\beta\int_{\mathbb{R}^{N}}C_1\left(|\gamma(v)|+|\gamma(v)|^{\theta-1}\right)|\gamma(v)|\gamma(v_T)^{2\beta}\nonumber\\
&\leq C_\beta C_1\left[\int_{\mathbb{R}^{N}}\gamma(vv_T^\beta)^{2}+\int_{\mathbb{R}^{N}}|\gamma(v)|^{\theta-2}\gamma(v)^2\gamma(v_T)^{2\beta}\right].
\end{align}

Applying Lemma \ref{c1}, inequality \eqref{rhs} becomes
\begin{align}\label{rhs2}
\|vv^\beta_T\|^2&\leq C_\beta C_1\left[\int_{\mathbb{R}^{N}}\gamma(vv_T^\beta)^{2}+\int_{\mathbb{R}^{N}}\left(1+g_2\right)\gamma(vv_T^\beta)^{2}\right]\nonumber\\
&\leq C_\beta C_1\left[2\int_{\mathbb{R}^{N}}\gamma(vv_T^\beta)^{2}
+\int_{\mathbb{R}^{N}}g_2\gamma(vv_T^\beta)^{2}\right],
\end{align}
where $g_2\in L^N(\mathbb{R}^{N})$.

Because $|\gamma(u)|_{2^\#}\leq C_{2^{\#}}\|u\|$ for all $u\in H^1(\mathbb{R}^{N+1}_+)$, it follows then from \eqref{norm=r} and \eqref{rhs2} that
\begin{align}\label{rhs3}
|\gamma(vv^{\beta}_T)|^2_{2^{\#}}&\leq C^2_{2^{\#}}C_\beta C_1\left[2\int_{\mathbb{R}^{N}}\gamma(vv_T^\beta)^{2}+\int_{\mathbb{R}^{N}}g_2\gamma(vv_T^\beta)^{2}\right].
\end{align}

Let us consider the last integral in the right-hand side of \eqref{rhs3}. For all $M>0$, define $A_1=\{g_2\leq M\}$ and $A_2=\{g_2>M\}$. Then, whereas $g_2\in L^N(\mathbb{R}^{N})$,
\begin{align*}
\int_{\mathbb{R}^{N}}g_2\gamma(vv_T^\beta)^{2}\leq& M\int_{A_1}\gamma(vv_T^\beta)^{2}+\left(\int_{A_2}g_2^N\right)^{\frac{1}{N}}\left(\int_{A_2}\gamma(vv_T^\beta)^{2\frac{N}{N-1}}\right)^{\frac{N-1}{N}}\nonumber\\
\leq&M\int_{\mathbb{R}^{N}}\gamma(vv_T^\beta)^{2}+\epsilon(M)\left(\int_{\mathbb{R}^{N}}\gamma(vv_T^\beta)^{2^{\#}}\right)^{\frac{N-1}{N}},
\end{align*}
and $\epsilon(M)=\left(\int_{A_2}g_2^N\right)^{1/N}\to 0$ when $M\to\infty$.

If $M$ is taken so that $\epsilon(M)C^2_{2^{\#}}C_\beta C_1<1/2$, we obtain
\begin{align}\label{rhs4a}
|\gamma(vv^{\beta}_T)|^2_{2^{\#}}&\leq K\int_{\mathbb{R}^{N}}\gamma(vv_T^\beta)^{2}=K|\gamma(vv^{\beta}_T|^2_2,
\end{align}
for a positive constant $K$ depending on $\beta$. Now, since $vv^\beta_T\to v^{1+\beta}_+$ when $T$ goes to infinity, it follows
\begin{align}\label{rhs4}
|\gamma(v^{1+\beta}_+)|^2_{2^{\#}}\leq &K|\gamma(v^{1+\beta}_+)|^2_2.
\end{align}
%
%


Choosing $\beta_1+1:=(\theta/2)>1$, it follows from \eqref{gammav} that the right-hand side of \eqref{rhs4} is finite. We conclude that $|\gamma(v_+)|\in {L^{2\frac{\theta}{2}}}(\mathbb{R}^{N})<\infty$. Now, we choose $\beta_2$ so that $\beta_2+1=(\theta/2)^2$ and conclude that
\[|\gamma(v_+)|\in L^{2\frac{\theta^2}{2^2}}(\mathbb{R}^{N}).\]

After $k$ iterations we obtain that
\[|\gamma(v_+)|\in L^{2\frac{\theta^k}{2^k}}(\mathbb{R}^{N}),\]
from what follows that $\gamma(v_+)\in L^p(\mathbb{R}^{N})$ for all $p\in [2,\infty)$. Since the same arguments are valid for $v_-$, we have $\gamma(v)\in L^p(\mathbb{R}^{N})$ for all $p\in [2,\infty)$.
$\hfill\Box$\end{proof}\vspace*{.2cm}

By simply adapting the proof given in \cite{ZelatiNolasco}, we present, for the convenience of the reader, the proof of our next result:\vspace*{.2cm}

\begin{proposition}\label{t2}
	Let $v\in H^1(\mathbb{R}^{N+1}_+)$ be a weak solution of \eqref{P}. Then $\gamma(v)\in L^p(\mathbb{R}^{N})$ for all $p\in [2,\infty]$ and $v\in L^\infty(\mathbb{R}^{N+1}_+)$.
\end{proposition}
\noindent\begin{proof} We recall equation \eqref{norm=r}:
\begin{align*}
\|vv^\beta_T\|^2\leq& C_\beta\int_{\mathbb{R}^{N}}f(\gamma(v))\gamma(v)\gamma(v_T)^{2\beta},
\end{align*}
where $C_\beta=\max\{m^{-2},(1+\beta^2)\}$. Now $|f(t)|\leq C_1(|t|+|t|^{\theta-1})$ yields
\begin{align*}
\|vv^\beta_T\|^2&\leq C_\beta C_1\left[\int_{\mathbb{R}^{N}}\gamma(vv_T^\beta)^ 2+\int_{\mathbb{R}^{N}}|\gamma(v)|^{\theta-2}\gamma(vv_T^{\beta})^2\right].
\end{align*}

Since $|\gamma(v)|^{\theta-2}=|\gamma(v)|^{\theta-2}\chi_{\{|\gamma(v)\leq 1\}}+|\gamma(v)|^{\theta-2}\chi_{\{|\gamma(v)> 1\}}$ and we know that $\gamma(v)\in L^p(\mathbb{R}^N)$ for all $p\geq 2$, we have $|\gamma(v)|^{\theta-2}\chi_{\{|\gamma(v)> 1\}}=:g_3\in L^{2N}(\mathbb{R}^{N})$. Thus,
\begin{align*}\gamma(vv_T^\beta)^ 2+ |\gamma(v)|^{\theta-2}\gamma(vv_T^{\beta})^2\leq (C_2+g_3)\gamma(vv_T^\beta)^2
\end{align*}
for a positive constant $C_2$ and a positive function $g_3\in L^{2N}(\mathbb{R}^{N})$ that depends neither on $T$ nor on $\beta$. Therefore,
\begin{align*}
\|vv^\beta_T\|^2&\leq C_\beta C_1\int_{\mathbb{R}^{N}}(C_2+g_3)\gamma(vv_T^\beta)^2,
\end{align*}
from what follows (when $T\to\infty$)
\begin{align*}
\|v^{\beta+1}_+\|^2&\leq C_\beta C_1\int_{\mathbb{R}^{N}}(C_2+g_3)\gamma(v^{\beta+1}_+)^2.
\end{align*}

From the inequality
\begin{align*}
\int_{\mathbb{R}^{N}}g_3\gamma(v^{\beta+1}_+)^2&\leq |g_3|_{2N}\,|\gamma(v^{1+\beta}_+)|_2\, |\gamma(v^{1+\beta}_+)|_{2^{\#}}\\
&\leq |g_3|_{2N}\left(\lambda|\gamma(v^{1+\beta}_+)|^2_2+\frac{1}{\lambda}|\gamma(v^{1+\beta}_+)|^2_{2^{\#}}\right),
\end{align*}
we conclude that
\begin{align}\label{esth1}
|\gamma(v^{1+\beta}_+)|^2_{2^{\#}}&\leq C^2_{2^{\#}} \|v^{\beta+1}_+\|^2\nonumber\\
&\leq C^2_{2^{\#}}C_\beta C_1\left(C_2+\lambda\,|g_3|_{2N}\right)|\gamma(v^{1+\beta}_+)|^2_2\nonumber\\
&\quad+\frac{C^2_{2^{\#}}C_\beta C_1\,|g_3|_{2N}}{\lambda}|\gamma(v^{1+\beta}_+)|^2_{2^{\#}}
\end{align}
and, by taking $\lambda>0$ so that
\[\frac{C^2_{2^{\#}}C_\beta C_1\,|g_3|_{2N}}{\lambda}<\frac{1}{2},\]
we obtain
\begin{align}\label{fest}
|\gamma(v^{1+\beta}_+)|^2_{2^{\#}}&\leq C_\beta\left(2C^2_{2^{\#}}C_2+2C^2_{2^{\#}}\lambda\,|g_3|_{2N}\right)|\gamma(v^{1+\beta}_+)|^2_2\nonumber\\
&\leq C_3C_{\beta}|\gamma(v^{1+\beta}_+)|^2_2.
\end{align}

Since
\[C_3C_\beta\leq C_3(m^{-2}+1+\beta)\leq M^2e^{2\sqrt{1+\beta}}\]
for a positive constant $M$, it follows from \eqref{fest} that
\begin{align*}
|\gamma(v_+)|_{2^{\#}(1+\beta)}&\leq M^{1/(1+\beta)}e^{1/\sqrt{1+\beta}}|\gamma(v_+)|_{2(1+\beta)}.
\end{align*}

We now apply an iteration argument, taking $2(1+\beta_{n+1})=2^{\#}(1+\beta_n)$ and starting with $\beta_0=0$. This produces
\[|\gamma(v_+)|_{2^{\#}(1+\beta_n)}\leq M^{1/(1+\beta_n)}e^{1/\sqrt{1+\beta_n}}|\gamma(v_+)|_{2(1+\beta_n)}.\]
Because $(1+\beta_n)=\left(\frac{2^{\#}}{2}\right)^n=\left(\frac{N}{N-1}\right)^n$,
we have
\[\sum_{i=0}^\infty \frac{1}{1+\beta_n}<\infty\qquad\textrm{and}\qquad \sum_{i=0}^\infty\frac{1}{\sqrt{1+\beta_n}}<\infty.\]

Thus,
\[|\gamma(v_+)|_\infty=\lim_{n\to\infty}|\gamma(v_+)|_{2^{\#}(1+\beta_n)}<\infty,\]
from what follows $|\gamma(v_+)|_p<\infty$ for all $p\in [2,\infty]$. The same argument applies to $\gamma(v_-)$, proving that $\gamma(v)\in L^p(\mathbb{R}^{N})$ for all $p\in [2,\infty]$.

By taking $\lambda=1$ and $|\gamma(v^{1+\beta}_+)|_{p}<C_4$ for all $p$ in \eqref{esth1}, we obtain for any $\beta>0$,
\begin{align}\label{final0}
\|v^{\beta+1}_+\|^2
&\leq C_\beta\left(C_3+|g_3|_{2N}\right)C^{2}_4+C_\beta\,|g_3|_{2N}C^{2}_5.
\end{align}

But $\|v_+\|^{1+\beta}_{2^*(1+\beta)}=\|v_+^{1+\beta}\|_{2^*}\leq C_{2^*}\|v_+^{1+\beta}\|$ and for a positive constant $\tilde{c}$ results from \eqref{final0}  that
\[\|v_+\|^{2(1+\beta)}_{2^*(1+\beta)}\leq \tilde{c}C_\beta C^{2(1+\beta)}_4.\]
Thus,
\[\|v_+\|_{2^*(1+\beta)}\leq \tilde{c}^{1/2(1+\beta)}C_\beta^{1/2(1+\beta)}C_4 \]
and the right-hand side of the last inequality is uniformly bounded for all $\beta>0$. We are done.
$\hfill\Box$\end{proof}

We now state a result obtained by Coti Zelati and Nolasco \cite[Proposition 3.9]{ZelatiNolasco}:
\begin{proposition}\label{regZN} Suppose that $v\in H^1(\mathbb{R}^{N+1}_+)\cap L^\infty(\mathbb{R}^{N+1}_+)$ is a weak solution of
\begin{equation}\label{C}\left\{\begin{aligned}
-\Delta v +m^2v&=0, &&\mbox{in} \ \mathbb{R}^{N+1}_+,\\
-\displaystyle\frac{\partial v}{\partial x}(0,y)&=h(y) &&\mbox{for all} \ y\in\mathbb{R}^{N},\end{aligned}\right.\end{equation}
where $h\in L^p(\mathbb{R}^{N})$ for all $p\in [2,\infty]$.

Then $v\in C^{\alpha}([0,\infty)\times\mathbb{R}^{N})\cap W^{1,q}((0,R)\times\mathbb{R}^{N})$ for all $q\in [2,\infty)$ and $R>0$.

In addition, if $h\in C^\alpha(\mathbb{R}^N)$, then $v\in C^{1,\alpha}([0,\infty)\times\mathbb{R}^N)\cap C^2(\mathbb{R}^{N+1}_+)$ is a classical solution of \textup{\eqref{C}}.
\end{proposition}

\noindent{\textit{Proof of Theorem \ref{classical}.} In the proof of Proposition \ref{regZN} (see \cite[Proposition 3.9]{ZelatiNolasco}), defining
	\[\rho(x,y)=\int_0^x v(t,y)\dd t,\]
	taking the odd extension of $h$ and $\rho$ to the whole $\mathbb{R}^{N+1}$ (which we still denote simply by $h$ and $\rho$), in \cite{ZelatiNolasco} is obtained that $\rho$ satisfies the equation
	\begin{equation}\label{rho}-\Delta\rho+m^2\rho=h\quad\text{in }\ \mathbb{R}^{N+1}\end{equation}
	and $\rho\in C^{1,\alpha}(\mathbb{R}^{N+1})$ for all $\alpha\in(0,1)$ by applying Sobolev's embedding. Therefore, $v(x,y)=\frac{\partial \rho}{\partial x}(x,y)\in C^\alpha(\mathbb{R}^{N})$.
	
In our case
\[h(y)=f\left(v(0,y)\right).\]

We now rewrite equation \eqref{rho} as
\[-\Delta \rho+m^2\rho=f\left(\frac{\partial \rho}{\partial x}(0,y)\right).\]
Since $f\in C^1$ and $\frac{\partial \rho}{\partial x}(x,y)$ is bounded, the right-hand side of the last equality belongs to $C^\alpha(\mathbb{R}^{N+1})$. Thus, classical elliptic boundary regularity yields
 \[\rho\in C^{2}(\mathbb{R}^{N+1})\quad\Rightarrow\quad v\in C^{1,\alpha}(\mathbb{R}^{N+1}_+).\]
Hence, by applying classical interior elliptic regularity directly to $v$, we deduce that $v\in C^{1,\alpha}(\mathbb{R}^{N+1}_+)\cap C^{2}(\mathbb{R}^{N+1}_+)$ is a classical solution of problem \eqref{P}. $\hfill\Box$

\section{Proof of Theorem \ref{t3}}
Let us consider a critical point $v \in H^{1}(\mathbb{R}^{N+1}_{+})$ of $J$. Then
\[\left\{\begin{array}{ll}
-\Delta v+m^2v=0 &\text{in}\ \mathbb{R}^{N+1}_+\\
v(0,y)= \gamma(v) \in L^{2}(\mathbb{R}^{N}), &y\in \mathbb{R}^{N}=\partial \mathbb{R}^{N+1}_+.
\end{array}\right.\]
	
Considering the Fourier transform with respect to the variable $y \in \mathbb{R}^{N}$ we obtain
\[\mathcal{F}(v)(x, \xi) = e^{\sqrt{\vert 2 \pi \xi \vert^{2} + m^2x}} \mathcal{F}(\gamma(v)))(\xi)
\]
and hence 
\[\displaystyle\sup_{y \in \mathbb{R}^{N}} \vert v(x,y) \vert \leq C\vert \gamma (v) \vert_{2}e^{-mx}.
\]
	
Since we have $\gamma(u) \in L^{q}(\mathbb{R}^{N})$ for any $q \in [2,\infty)$ as a consequence of Proposition \ref{t2},  we have that $u(x,y) \to 0$ as $\vert y \vert \to \infty$ for any $x$ and conclude that $u(x,y)e^{\lambda x} \to 0$, as $x + \vert y \vert \to \infty$, for any $0< \lambda < m$.

%
%

\textit{Proof of Theorem \ref{t3}.} Applying the strong maximum principle and Hopf lemma, we have $u(x,y) > 0$ for any $(x,y) \in \mathbb{R}^{N+1}_{+}$.
	
For $R>0$ let us define
\begin{align*}
B_{R}^{+} &= \{(x,y) \in \mathbb{R}^{N+1}_{+}; \,\, \sqrt{x^2 +\vert y \vert^{2}} < R\}\\
\Omega_{R}^{+} &= \{(x,y) \in \mathbb{R}^{N+1}_{+}; \,\, \sqrt{x^2 +\vert y \vert^{2}} > R\}\\
\Gamma_{R} &= \{(0,y) \in \mathbb{R}^{N+1}_{+}; \,\,  \vert y \vert \geq R \}\\
\end{align*}
and the auxiliary function
\[f_R(x,y) = C_{R}e^{-\alpha x} e^{-(m-\alpha)\sqrt{x^2 + \vert y \vert^2}}, \,\, \, \text{for} \, \, (x,y) \in \Omega_{R}^{+}
\]
with $0\leq \alpha < m$ and $C_{R}>0$ a constant to be fixed later.
	
We have
\begin{align*}
-\Delta f_{R} + m^2f_{R} &= \left[\alpha^{2} + m^2 + (m-\alpha)^{2} + \frac{N(m-\alpha)}{\sqrt{x^2 + \vert y \vert^2}} - \frac{2\alpha(m-\alpha)x}{\sqrt{x^2 + \vert y \vert^2}} \right]f_{R} \\
&\geq \frac{N(m-\alpha)}{\sqrt{x^2 + \vert y \vert^2}}f_{R} \geq 0\,\,\, \text{in} \,\,\, \Omega_{R}^{+}.
\end{align*}
	
	Moreover, 
\[\frac{\partial f_{R}}{\partial \eta}(0,y) = -\frac{\partial f_{R}}{\partial x}(0,y) = \alpha f_{R}(0,y) \, \,\, \text{in} \,\,\, \Gamma_{R}.
\]
	
Let us define $w(x,y) = f_{R}(x,y) - v(x,y)$, for $(x,y) \in \overline{\Omega}_{R}^{+}$. Thus,
\begin{align*}
-\Delta w + m^{2} w &= -\Delta f_{R} + m^2f_{R} - \displaystyle\underbrace{-\Delta u + m^2 u}_{=0} \\
&= -\Delta f_{R} + m^2f_{R} \geq 0 \,\,\, \text{in} \,\,\, \overline{\Omega}_{R}^{+}
\end{align*}
and choosing $C_{R} = e^{mR} \displaystyle\max_{\partial B_{R}^{+}} u$ we get for $(x,y) \in \partial B_{R}^{+}$,
\[w(x,y)= f_{R}(x,y) - v(x,y)  \,\,\geq \,\, \displaystyle\max_{\partial B_{R}^{+}} u - v(x,y) \geq 0, \quad \quad \forall (x,y) \in \partial B_{R}^{+}.
\]
	
In addition, since we already know that $u(x,y) \to 0$ when  $x + \vert y \vert \to \infty$ and the same is true for $f_R(x,y)$, we conclude that $w(x,y) \to 0$ as $x + \vert y \vert \to \infty$.

\noindent\textbf{Claim.} We have $w(x,y) \geq 0$ in $\Omega_{R}^{+}$.
		
In fact, suppose that
\[\displaystyle\inf_{\overline{\Omega}_{R}^{+}} w < 0.
\]
	
By the strong maximum principle, there exists $(0,y_0) \in \Gamma_{R}$ such that
\[w(0,y_0) = \displaystyle\inf_{\overline{\Omega}_{R}^{+}} w \leq w(x,y), \,\,\,\,\, \forall (x,y) \in \Omega_{R}^{+}.
\]
	
Now, we define $W(x,y) = w(x,y)e^{\lambda x}$, with $\lambda \in (0,m)$. Thus,
\[W(x,y) =C_{R}e^{(\lambda - \alpha)x}e^{-(m-\alpha)\sqrt{x^2 + \vert y \vert^2}} - e^{\lambda x} v(x,y)
\]
and, as before, we have $W(x,y) \to 0$ as $x + \vert y \vert \to \infty$ and $W(x,y) \geq 0$ in $\partial B_{R}^{+}$.
	
Note that,
\[-\Delta w + m^2 w = e^{-\lambda x}\left(-\Delta W + 2 \lambda \frac{\partial W}{\partial x} + (m^2 - \lambda^2) W\right),	
\]
thus yielding $-\Delta W + 2 \lambda \frac{\partial W}{\partial x} + (m^2 - \lambda^2) W \geq 0$ on $\Omega_{R}^{+}$.
	
By the strong maximum principle,
\[\displaystyle\inf_{\Gamma_{R}} = \displaystyle\inf_{\overline{\Omega}_{R}^{+}} < W(x,y) , \,\,\,\,\,\, (x,y) \in \Omega_{R}^{+}.
\]
	
Therefore,
\[W(0,y_0) = \displaystyle\inf_{\Gamma} W = \displaystyle\inf_{\Gamma} w = w(0,y_0) < 0.
\]
	
It follows from Hopf's lemma that
\begin{align}\label{Hopfl}
-\frac{\partial W}{\partial x}(0,y_0) = \frac{\partial W}{\partial \eta}(0,y_0) < 0.
\end{align}

But
\begin{align*}
-\frac{\partial W}{\partial x}(0,y_0) &=  -\frac{\partial w}{\partial x}(0,y_0) - \lambda w(0,y_0) \\
&= -\frac{\partial f_R}{\partial x}(0,y_0) + \frac{\partial u}{\partial x}(0,y_0) - \lambda f_{R}(0,y_0) + \lambda v(0,y_0) \\
&= \lambda v(0,y_0) - f(v(0,y_0)) + (\alpha - \lambda) f_{R}(0,y_0).
\end{align*}
	
Since $\vert y_0 \vert \to \infty$ as $R \to \infty$ and $v(0,y) \to 0$ as $\vert y \vert \to \infty$, follow from $(f_1)$ and $(f_2)$ that $f(v(0,y_0)) \to 0$ as $R \to \infty$. Thus, for any $0 < \lambda < \alpha < m$ and $R$ large enough we have
\[-\frac{\partial W}{\partial x}(0,y_0) \lambda v(0,y_0) - f(v(0,y_0)) + (\alpha - \lambda) f_{R}(0,y_0) > 0
\]
a contradiction with \eqref{Hopfl}.
	
Therefore, $w(x,y) \geq 0$ in $\overline{\Omega}_{R}^{+}$ and thus
\begin{align*}
0 < v(x,y)&= f_{R}(x,y) - w(x,y)\\
&\leq f_{R}(x,y) = C_{R}e^{-\alpha x}e^{-(m-\alpha)\sqrt{x^2 + \vert y\vert^2}}
\end{align*}
for all $(x,y) \in \overline{\Omega}_{R}{+}$.
	
In particular, for $\delta = m-\alpha > 0$ we finally obtain
\[0<v(0,y) \leq Ce^{-\delta \vert y \vert}, \,\, \text{for any} \,\, \vert y \vert \geq R,
\]
and we are done.
$\hfill\Box$

\section{Radial solution}
In this section we will prove that two different problems have radially symmetric solutions. The proof of our results adapt ideas of Choi and Seok \cite[Proposition 4.2]{ChoiSeok}. We initially  consider the problem \eqref{P}.


\textit{Proof of Theorem \ref{t4}.} By applying Theorems \ref{classical} and \ref{t3}, any solution $v$ of \eqref{P} is regular and satisfies
	\[\lim_{x+|y|\to \infty}v(x,y)=0.\]
	
We now apply the moving planes method together with the maximum principle. For any $\lambda>0$ we define
\begin{align*}
R_\lambda&=\{(x,y)=(x,y_1,\ldots,y_n)\in \mathbb{R}^{N+1}_+\,:\,x\geq 0,\,y_1>\lambda \},\\
\Sigma_\lambda&=\{y\in\mathbb{R}^N\,:\, (x,y)\in R_\lambda\},\\
y^\lambda&=(2\lambda-y_1,y_2,\ldots,y_n),\\
v_\lambda(x,y)&=v(x,y^\lambda).
\end{align*}
	
Note that $\Sigma_\lambda$ is the projection of $R_\lambda$ on $\mathbb{R}^N$. Denoting $w_\lambda=v_\lambda- v$, we have
\begin{equation}\label{rad2}
\left\{\begin{array}{rcll}
-\Delta w_\lambda+m^2 w_\lambda&=&0, &\text{in } \mathbb{R}^{N+1}_+\\
\displaystyle-\frac{\partial w_\lambda}{\partial x}(0,y)&=&f(v_\lambda(0,y))- f(v(0,y)), &y\in\mathbb{R}^N.
\end{array}\right.
\end{equation}

\noindent\textit{Claim.}
For $\lambda>0$ large enough, we have $w_{\lambda} \geq 0$ in $R_{\lambda}$.

In fact, define $w_{\lambda}^{-} = \min\{0,w_{\lambda}\}$ and consider
\begin{equation*}
c^\lambda(y)=\left\{\begin{array}{ll}\displaystyle\frac{f(v_\lambda(0,y))-f(v(0,y))}{v_\lambda(0,y)-v(0,y)},\ &\text{if }\ v_\lambda(0,y)\neq v(0,y),\\
0,\ &\text{if }\ v_\lambda(0,y)= v(0,y),\end{array}\right.
\end{equation*}

As a consequence of ($f_1$) we have $f'(t) \to 0$ as $t \to 0$ and thus $c^{\lambda}(y)$ converges to $0$ uniformly on $R_{\lambda}$ when $\lambda\to\infty$.

Taking $w^{-}_\lambda$ as a test-function in \eqref{rad2} yields
\begin{equation}\label{Iden1}
\iint_{\mathbb{R}^{N+1}_{+}}\left[|\nabla w^{-}_\lambda|^2+m^2|w^{-}_\lambda|^2\right]= \int_{\mathbb{R}^{N}} \left(f(v_{\lambda}(0,y)) - f(v(0,y)) \right)w^{-}_{\lambda} \dd y.
\end{equation}
	
Now, observe that the change variable $y \to y^{\lambda}$ yields
\begin{align*}
\iint_{\mathbb{R}^{N+1}_{+}}\left[|\nabla w^{-}_\lambda|^2+m^2|w^{-}_\lambda|^2\right]= 2 \iint_{R_\lambda}\left[|\nabla w^{-}_\lambda|^2+m^2|w^{-}_\lambda|^2\right]
\end{align*}
and also
\begin{align*}
\int_{\mathbb{R}^{N}} \left(f(v_{\lambda}(0,y)) - f(v(0,y)) \right)w^{-}_{\lambda} \dd y &=  \int_{\mathbb{R}^{N}}c^{\lambda}(y) \vert w^{-}_{\lambda}(0,y)\vert^{2}  \dd y \nonumber\\
&= 2\int_{\Sigma_{\lambda}} c^{\lambda}(y) \vert w^{-}_{\lambda}(0,y)\vert^{2}  \dd y.
\end{align*}
	
Substituting the last two equalities 
in \eqref{Iden1} we obtain
\begin{align}\label{rlambdasigmalambda}
\iint_{R_\lambda}\left[|\nabla w^{-}_\lambda|^2+m^2|w^{-}_\lambda|^2\right]&= \int_{\Sigma_{\lambda}} c^{\lambda}(y) \vert w^{-}_{\lambda}(0,y)\vert^{2}  \dd y.
\end{align}
	
Since $ c^{\lambda}(y) \to 0$ uniformly in $\Sigma_{\lambda}$ when $\lambda\to\infty$, it follows that, for $\lambda$ large enough, $\vert c~{\lambda}(y) \vert \leq \varepsilon$ for any $\varepsilon > 0$. Thus, it follows from Lemma \ref{lemmameas} that
\begin{align*}
\iint_{R_\lambda}\left[|\nabla w^{-}_\lambda|^2+m^2|w^{-}_\lambda|^2\right] &\leq \varepsilon\int_{\Sigma_{\lambda}} \vert w^{-}_{\lambda}(0,y)\vert^{2}  \dd y\\
&\leq \varepsilon \iint_{\Sigma_\lambda \times [0,\infty)}\left[|\nabla w^{-}_\lambda|^2+m^2|w^{-}_\lambda|^2\right]\\
&= \varepsilon \iint_{R_\lambda}\left[|\nabla w^{-}_\lambda|^2+m^2|w^{-}_\lambda|^2\right],
\end{align*}
allowing us to conclude that $w^{-}_{\lambda} = 0$ in $R_{\lambda}$ for $\lambda$ large enough, that is, $w_{\lambda} \geq 0$ in $R_{\lambda}$, proving our claim.
	
Now we define
\[\nu=\inf\{\sigma>0\,:\,w_\lambda\geq 0\ \text{on }\ R_\lambda,\ \forall  \lambda  \geq \sigma\}.\]

We start considering the case $\nu>0$ and claim that, in this case, we have $w_\nu\equiv 0$ on $R_\lambda$. If not, it follows from the continuity of $w_{\nu}$ and the strong maximum principle that $w_\nu>0$ on the set
\[R'_\nu=\{(x,y)=(x,y_1,\ldots,y_n)\,:\, y_1>\nu,\ x>0\}.\] 
	
We now assert that $w_\nu > 0$ on the set $\{y\in\mathbb{R}^N\,:\,y_1>\nu\}$. Otherwise, there exists $\bar{y}\in\mathbb{R}^N$ such that $w_\nu(\bar{y})=0$, with its first coordinate greater than $\nu$ . By the Hopf lemma we have $-\frac{\partial}{\partial x}w_\nu(0,\bar{y})>0$ and we have reached a contradiction, since $v_\nu(0,\bar{y})=v(0,\bar{y})$ and
\[-\frac{\partial}{\partial x}w_\nu(0,\bar{y})= f(v_{\nu}(0,\bar{y})) - f(v(0,\nu)) = 0.
\]
Thus $w_\nu > 0$ on the set $\{y\in\mathbb{R}^N\,:\,y_1>\nu\}$.	

In order to reach a contradiction with the definition of $\nu$ if $\nu>0$, consider a sequence $\lambda_j<\nu$ such that $\lambda_j\to\nu$ when $j\to\infty$. Since $c^{\lambda_{j}} \to 0$ uniformly in $\Sigma_{\lambda}$ for $\lambda$ large enough, we have that $\vert c^{\lambda_{j}}(y) \vert \leq \varepsilon < \frac{1}{\gamma_0}$, for a positive constant $\varepsilon$ and any $\vert y \vert > r_0>0$. 

Let $D = \Vert c_{\lambda_j} \Vert_{L^{\infty}(\mathbb{R}^{N})} < \infty$ and $B_{r_0}(p_j)$ be the open ball with center $p_0=(\lambda_j,0,\ldots,0)\in\mathbb{R^N}$ and radius $r_0>0$. Then, according to \eqref{rlambdasigmalambda} and Lemma \ref{lemmameas},
\[\iint_{R_{\lambda_j}}\left[|\nabla w^{-}_{\lambda_j}|^2+ m^2|w^{-}_{\lambda_j}|^2\right]\]
\begin{align*}
&= \int_{\Sigma_{\lambda_j}\cap B_{r_0}(p_j)}c^{\lambda_j}(y)|w^{-}_{\lambda_j}(0,y)|^2\dd y +	\int_{\Sigma_{\lambda_j}\setminus B_{r_0}(p_j)}c^{\lambda_j}(y)|w^{-}_{\lambda_j}(0,y)|^2\dd y \\
&\leq  D  \int_{\Sigma_{\lambda_j}\cap B_{r_0}(p_j)}|w^{-}_{\lambda_j}(0,y)|^2\dd y + \varepsilon\int_{\Sigma_{\lambda_j}\setminus B_{r_0}(p_j)}|w^{-}_{\lambda_j}(0,y)|^2\dd y \\
&\leq D  \int_{\Sigma_{\lambda_j}\cap B_{r_0}(p_j)}|w^{-}_{\lambda_j}(0,y)|^2\dd y + \varepsilon \iint_{R_{\lambda_j}}\left[|\nabla w^{-}_{\lambda_j}|^2+m^2|w^{-}_{\lambda_j}|^2\right],
\end{align*}
and we conclude that
\begin{equation*}
\iint_{R_{\lambda_j}}\left[|\nabla w^{-}_{\lambda_j}|^2+m^2|w^{-}_{\lambda_j}|^2\right] \leq C \int_{\Sigma_{\lambda_j}\cap B_{r_0}(p_j)}|w^{-}_{\lambda_j}(0,y)|^2\dd y.
\end{equation*}
	
Denote by $E_{j}$ the set $\textrm{supp}\,w^{-}_{\lambda_{j}}(0,y)$ in $B_{r_0}(p_j)$. Since $w^{-}_{\nu} >0$ in $\Sigma_{\nu}$ and $\lambda_{j} \to \nu$, the continuity of $w_{\nu}$ yields that $\vert E_{j} \vert$ converges to $0$ as $j \to \infty$, since $w^{-}_{\lambda_{j}}(0,y) \to w^{-}_{\nu}(0,y) = 0$ in $\Sigma_{\nu}$. Thus, the dominated convergence theorem and H\"{o}lder's inequality imply that
\begin{align}\label{Iden7}\int_{\Sigma_{\lambda_j}\cap B_{r_0}(p_j)}|w^{-}_{\lambda_j}(0,y)|^2\dd y&=\int_{\Sigma_{\lambda_{j}}} \chi_{E_j}(y)\vert w^{-}_{\lambda_{j}}(0,y) \vert^{2} \dd y \nonumber\\
&= \vert E_{j} \vert^{\frac{1}{N}} \Vert w^{-}_{\lambda_{j}} \Vert_{L^{\frac{2N}{N-1}}(\Sigma_{\lambda_{j}})}\\
&\leq \vert E_{j} \vert^{\frac{1}{N}} \iint_{R_{\lambda_{j}}} \vert \nabla w^{-}_{\lambda_{j}}\vert^{2}\nonumber
\end{align}
	
Consequently, $w^{-}(x,y) = 0$ in $R_{\lambda_{j}}$, that is,
$w_{\lambda_{j}} \geq 0$ in $R_{\lambda_{j}}$,
contradicting the  definition of $\nu$. Thus, obtain $w_{\nu} = 0$ in $R_{\nu}$ and we obtain  the symmetry in the $y_1$ direction with respect to $y_1=\nu$.
	
If $\nu = 0$, we repeat the previous arguments for $\lambda < 0$ and $w_{\lambda} = v_{\lambda} - v$ defined on
\begin{equation*}
Q_{\lambda} = \left\{(x,y)=(x,y_1,\ldots,y_n)\in \mathbb{R}^{N+1}_+\,:\,y_1 <\lambda,\, x\geq 0\right\}.
\end{equation*}

Thus, as before, we conclude that $w_{\lambda} \geq 0$ when  $\lambda\to-\infty$. Define
\[\nu'=\sup\{\sigma<0\,:\,w_\lambda\geq 0\ \text{on }\ Q_\lambda,\ \forall  \lambda  \leq \sigma\}.\]
	
If $\nu < 0$, the preceding discussion applies and we obtain the symmetry with respect to $y_1=\nu'$. If $\nu' = 0$, we have
\begin{equation}\label{Iden4}
v(x,-y_1,y_2,\cdots,y_N) \leq v(x,y_1,\cdots,y_N) \quad \text{in} \quad Q_{0},
\end{equation}
and since $\nu = 0$ we have also
\begin{equation}\label{Iden5}
v(x,-y_1,y_2,\cdots,y_N) \leq v(x,y_1,\cdots,y_N) \quad \text{in} \quad R_{0}.
\end{equation}
	
From \eqref{Iden4} and \eqref{Iden5} follows that 
\begin{equation}\label{Iden6}
v(x,-y_1,y_2,\cdots,y_N) \leq v(x,y_1,\cdots,y_N) \quad \text{in} \quad \mathbb{R}^{N+1}_{+}.
\end{equation}
Consequently, replacing $y_1$ by $-y_1$ in \eqref{Iden6}, we obtain the symmetry with respect to $y_1=0$:
\begin{equation*}
v(x,-y_1,y_2,\cdots,y_N) = v(x,y_1,\cdots,y_N) \quad \text{in} \quad \mathbb{R}^{N+1}_{+}.
\end{equation*}

To conclude the proof we apply the same procedure with respect to the other directions $y_i$, for $i=2,...,N$.  
$\hfill\Box$\vspace*{.3cm}

We now consider the problem \eqref{Hartree}.
For $y\in\mathbb{R}^N$ we denote 
\[g(y)=\int_{\mathbb{R}^N}\frac{1}{|y-z|^{N-\alpha}}F(u(0,z))\dd z.\]


\textit{Proof of Theorem \ref{t5}.}
We maintain the notation introduced in the proof of Theorem \ref{t4} and define $g_\lambda(y)=g(y^\lambda)$. Observe that, as before, any solution of \eqref{Hartree} is regular and satisfies
	\[\lim_{x+|y|\to \infty}v(x,y)=0,\]
see e.g. \cite{BBMP}.
	
We now apply the moving planes method in integral form. 
As before, $w_\lambda$ stands for $v_\lambda-v$. Then we have, for $y\in\mathbb{R}^N$,
	\begin{equation}\label{pradsv3}
	\left\{\begin{array}{rcl}
	-\Delta w_\lambda+m^2 w_\lambda&=&0,\qquad \text{in } \mathbb{R}^{N+1}_+\\ 
	\displaystyle-\frac{\partial w_\lambda}{\partial x}(0,y)&=&-V_\infty w_\lambda(0,y)+g_\lambda(y)f(v_\lambda(0,y))-g(y)f(v(0,y)). 
	\end{array}\right.
	\end{equation}
	
We claim that $w_\lambda\leq 0$ on $R_\lambda$ for $\lambda$ large enough. To prove our claim, we define $w^+_\lambda(x,y)=\max\{0,w_\lambda\}$ and consider
\begin{align*}
c^\lambda_1(x,y)=\left\{\begin{array}{ll}\displaystyle\frac{f(v_\lambda)-f(v)}{v_\lambda-v},\ &\text{if }\ v_\lambda\neq v,\\
0,\ &\text{if }\ v_\lambda= v,\end{array}\right.\\
\intertext{and}
c^\lambda_2(x,y)=\left\{\begin{array}{ll}\displaystyle\frac{F(v_\lambda)-F(v)}{v_\lambda-v} &\text{if } v_\lambda\neq v,\\
0,\ &\text{if }\ v_\lambda= v.
\end{array}\right.
\end{align*}
	
Observe that, when $\lambda\to\infty$, both $c^\lambda_1(x,y)$ and $c^\lambda_2(x,y)$ converge to $0$ uniformly on $R_\lambda$.
	
Taking $w^+_\lambda$ as a test-function in \eqref{Hartree}, the same argument applied to obtain \eqref{rlambdasigmalambda} yields
\[\iint_{R_\lambda}\left[|\nabla w^+_\lambda|^2+m^2|w^+_\lambda|^2\right]+V_\infty\int_{\Sigma_\lambda} w^+_\lambda(0,y)\dd y\]
	\begin{align}\label{c2}&=\int_{\Sigma_\lambda}\left[g_\lambda(y)f(v_\lambda(0,y))-g(y)f(v(0,y))\right]w^+_\lambda(0,y)\dd y\nonumber\\
	&=\int_{\Sigma_\lambda}g_\lambda(y)c^\lambda_1(x,y)|w^+_\lambda(0,y)|^2\dd y+\int_{\Sigma_\lambda}[g_\lambda(y)-g(y)]f(v(0,y))w^+_\lambda(0,y)\dd y\\
	&=I_1+I_2,\nonumber
	\end{align}
	respectively. 
	
	We now consider the integral $I_1$ in right-hand side of \eqref{c2}. Since $g\in L^\infty(\mathbb{R}^N)$ and $c^\lambda_1\to 0$ uniformly when $\lambda\to\infty$, we have $c^\lambda_1(0,y)g_\lambda(y)=\epsilon(\lambda)$, where $\epsilon(\lambda)\to 0$ when $\lambda$ is large enough. Thus
	\begin{equation}\label{c4}\int_{\Sigma_\lambda}g_\lambda(y)c^\lambda_1(x,y)|w^+_\lambda(0,y)|^2\dd y=\epsilon(\lambda)\int_{\Sigma_\lambda}|w^+_\lambda(0,y)|^2\dd y.
	\end{equation}
	
	We now consider $I_2$. Since $F$ is increasing, we have
	\begin{subequations}
		\begin{align}
		g_\lambda(y)-g(y)&=\int_{\Sigma_\lambda}\left[\frac{1}{|y-z|^{N+\alpha}}-\frac{1}{|y-z^\lambda|^{N+\alpha}}\right][F(v_\lambda(0,z))-F(v(0,z))]\dd z\nonumber\\
		&\leq\int_{\Sigma_\lambda\cap \{w^\lambda(0,\cdot)>0\}}\frac{1}{|y-z|^{N+\alpha}}\left[F(v_\lambda(0,z)-F(v(0,z))\right]\dd z\nonumber\\
		&\quad+\int_{\Sigma_\lambda\cap \{w^\lambda(0,\cdot)<0\}}\frac{1}{|y-z|^{N+\alpha}}\left[F(v_\lambda(0,z)-F(v(0,z))\right]\dd z\label{final1}\\
		&\leq\int_{\Sigma_\lambda}\frac{1}{|y-z|^{N+\alpha}}|c^\lambda_2(0,z)|\,w^+_\lambda(0,z)\dd z,\label{final2}
		\end{align}
	\end{subequations}
since the second integral in \eqref{final1} is negative and can be ignored, while the first integral in \eqref{final1} is bounded by the integral in \eqref{final2}. So, 
	\begin{align*}\label{c5}
	I_2&\leq \int_{\Sigma_\lambda}\int_{\Sigma_\lambda}\frac{1}{|y-z|^{N+\alpha}}|c^\lambda_2(0,z)|\,|w^+_\lambda(0,z)|\,|f(v(0,y))w^+_\lambda(0,y)|\dd z\dd y\nonumber\\
	&\leq K|c
	^\lambda_2(0,\cdot)|_{L^\infty(\Sigma_\lambda)}|w^+_\lambda(0,\cdot)|_{L^2(\Sigma_\lambda)}\,|f(v(0,\cdot))w^+_\lambda(0,\cdot)|_{L^q(\Sigma_\lambda)},
	\end{align*}
	as a consequence of the Hardy-Littlewood inequality, with
	\[\frac{1}{q}=2-\frac{1}{2}-\frac{N-\alpha}{N}=\frac{1}{2}+\frac{\alpha}{N}>\frac{1}{2},\]
	the constant $K$ depending on $q$. Observe that we have $q<2$. 
	
	Thus, it follows from Hölder's inequality that 
	
	\begin{align}\label{c6}
	I_2&\leq O(\lambda)\int_{\Sigma_\lambda}|w^+_\lambda(0,y)|^2\dd y,
	\end{align}
	where 
	\[O(\lambda)=K|c^\lambda_2(0,\cdot)|_{L^\infty(\Sigma_\lambda)}|f(v(0,\cdot))|_{L^\delta(\Sigma_\lambda)}\to 0\quad\text{when }\ \lambda\to \infty\]
	and $\delta=2q/(2-q)$.
	
Returning to \eqref{c2}, we conclude that
\[\iint_{R_\lambda}\left[|\nabla w^+_\lambda|^2+m^2|w^+_\lambda|^2\right]+V_\infty\int_{\Sigma_\lambda} w^+_\lambda(0,y)\dd y\]
\begin{align}\label{c7}
&\leq \left[\epsilon(\lambda)+O(\lambda)\right]\int_{\Sigma_\lambda}|w^+_\lambda(0,y)|^2\dd y\nonumber\\
&\leq \left[\epsilon(\lambda)+O(\lambda)\right]\iint_{R_\lambda}\left[|\nabla w^+_\lambda|^2+m^2 |w^+_\lambda|^2\right],
\end{align}
as a consequence of Lemma \ref{lemmameas}. Thus, 
$w^+_\lambda\equiv 0$ on $R_\lambda$ for $\lambda$ large enough, proving our claim. 
	
Now we define
\[\nu=\inf\{\sigma>0\,:\,w_\lambda\leq 0\ \text{on }\ R_\lambda,\ \forall \sigma\leq \lambda\}.\]
	
Let us start considering the case $\nu>0$. We claim that $w_\nu\equiv 0$ on $R_\lambda$. If not, as a consequence of the strong maximum principle, we have $w_\nu<0$ on the set
\[R'_\nu=\{(x,y)=(x,y_1,\ldots,y_n)\,:\, y_1>\nu,\ x>0\},\]
since $w_\nu\leq 0$ on $R_\nu$ is valid by continuity. We assert that $w_\nu<0$ on the set $\{y\in\mathbb{R}^N\,:\,y_1>\nu\}$. Otherwise, there exists $\bar{y}\in\mathbb{R}^N$ such that $w_\nu(\bar{y})=0$, with $\bar{y}_1>\nu$. By the Hopf Lemma, we have $-\frac{\partial}{\partial x}w_\nu(0,\bar{y})>0$ and we have reached a contradiction, because
\[-\frac{\partial}{\partial x}w_\nu(0,\bar{y})=g_\lambda(\bar{y})f(v_\nu(0,\bar {y}))-g(\bar{y})f(v(0,\bar{y}))=\left[g_\nu(\bar{y})-g(\bar{y})\right]f(v_\nu(0,\bar{y}))<0,
\]
since $v_\nu(0,\bar{y})=v(0,\bar{y})$ and $1/|x-y^\lambda|\leq 1/|x-y|$.
	
In order to reach a contradiction with the definition of $\nu$, consider a sequence $\lambda_j<\nu$ such that $\lambda_j\to\nu$ when $j\to\infty$. Let $B_{r_0}(p_j)$ be the ball with center $p_0=(\lambda_j,0,\ldots,0)\in\mathbb{R^N}$ and radius $r_0>0$. Then, according to \eqref{c7},
\[\iint_{R_{\lambda_j}}\left[|\nabla w^+_{\lambda_j}|^2+m^2|w^+_{\lambda_j}|^2\right]+V_\infty\int_{\Sigma_{\lambda_j}} w^+_{\lambda_j}(0,y)\dd y
\]
\begin{align}
&\leq M\left(\int_{\Sigma_{\lambda_j}\cap B_{r_0}(p_j)}|w^+_{\lambda_j}(0,y)|^2\dd y +	\int_{\Sigma_{\lambda_j}\setminus B_{r_0}(p_j)}|w^+_{\lambda_j}(0,y)|^2\dd y\right). 
\end{align}
	
Since $w^+_{\lambda_j}$ has exponential decay, by taking $r_0$ large enough we have that \[M\int_{\Sigma_{\lambda_j}\setminus B_{r_0}(p_j)}|w^+_{\lambda_j}(0,y)|^2\dd y\leq \frac{1}{2}\iint_{R_{\lambda_j}} \vert \nabla w^+_{\lambda_j}\vert^{2}\dd y.
\]
	
Thus, 
\[\iint_{R_{\lambda_j}}\left[|\nabla w^+_{\lambda_j}|^2+m^2|w^+_{\lambda_j}|^2\right]+V_\infty\int_{\Sigma_{\lambda_j}} w^+_{\lambda_j}(0,y)\dd y\hfill \]
\begin{align}
\leq M_{0} \int_{\Sigma_{\lambda_j}\cap B_{r_0}(p_j)}|w^+_{\lambda_j}(0,y)|^2\dd y.
\end{align}
	
The same arguments applied to obtain \eqref{Iden7} yield $w_{\lambda_{j}} \leq 0$ in $R_{\lambda_{j}}$, contradicting the definition of $\nu$. Thus $w_{\nu} = 0$ on $R_{\nu}$  and the symmetry in the $y_1$ direction follows.
	
If $\nu=0$, we also repeat the arguments in the proof of Theorem \ref{t4} to conclude that $u$ is symmetric in the $y_1$ direction.
$\hfill\Box$

\textbf{Acknowledgements:} Aldo H. S. Medeiros received a grant by CNPq - Brasil.

\end{document}